\documentclass[oneside,english]{article}

 \usepackage{amssymb,latexsym,amsmath,epsfig,amsthm, amsfonts}
\theoremstyle{plain}
\makeatletter

\renewcommand\section{\@startsection {section}{1}{\z@}
{-30pt \@plus -1ex \@minus -.2ex}
{2.3ex \@plus.2ex}
{\normalfont\normalsize\bfseries}}

\renewcommand\subsection{\@startsection{subsection}{2}{\z@}
{-3.25ex\@plus -1ex \@minus -.2ex}
{1.5ex \@plus .2ex}
{\normalfont\normalsize\bfseries}}

\renewcommand{\@seccntformat}[1]{\csname the#1\endcsname. }

\makeatother

\newtheorem{theorem}{Theorem}

\newtheorem{corollary}{Corollary}

\newtheorem{remark}{Remark}
\newtheorem{example}{Example}

\def\stirling#1#2{\genfrac{\{}{\}}{0pt}{}{#1}{#2}}

\title{Restricted lonesum matrices}
\author{Be\'{a}ta B\'{e}nyi\\
\small Faculty of Water Sciences, National University of Public Service, Budapest\\
\small\tt beata.benyi@gmail.com\\
}
\date{\today\\%\dateline{Jan 1, 9999}{Mar 25, 2016}\\
}
\begin{document}
\maketitle
\setlength{\unitlength}{0.5cm}
\begin{abstract}
Lonesum matrices are matrices that are uniquely reconstructible from their row and column sum vectors. These matrices are enumerated by the poly-Bernoulli numbers; a sequence related to the multiple zeta values with a rich literature in number theory. Lonesum matrices are in bijection with many other combinatorial objects: several permutation classes, matrix classes, acyclic orientations in complete bipartite graphs etc. Motivated of these facts, we study in this paper lonesum matrices with restriction on the number of columns and rows of the same type. 
\end{abstract}
\section{Introduction}
%Stirling numbers of the second kind is a basic and important sequence in combinatorics. $\stirling n m $ gives the number of partitions of an $n$ element set into exactly $m$ nonempty blocks (that is the reason why these numbers are sometimes also called Stirling partition numbers). These numbers are crucial in many combinatorial enumerations; hence, it arises often in formulas, famous/significant sequences as Bell numbers, Fubini numbers (surjection numbers), poly-Bernoulli numbers, etc. 

%Restrictions on the blocks of the set partition lead to modifications of the Stirling numbers. The restriction can affect the size of the blocks or particular elements of blocks (which elements can arise in one block). 

Lonesum matrices are $01$ matrices that are uniquely reconstructible from their row and column sum vectors. They play in discrete tomography an important role. Discrete tomography is dealing with the retrieval of information about objects from the data about its projections. One of the first problems that was discussed was the problem of reconstruction of a $01$ matrix from its column and row vectors \cite{Ryser}. 

Recently, an interesting application arose in systems biology. One of the fundamental question in systems biology is how a small number of signaling inputs specifies a large number of cell fates through the coordinated expression of thousand of genes.  In \cite{Biol} the authors used sequential logic to explain biological networks. The authors reduced the problem of determining the number of accessible configurations by introducing the connectivity matrix representation of a network to the problem of finding the number of matrices satisfying certain patterns. It turned out that the connectivity matrix of the ratchet model is a lonesum matrix. 

One of the crucial property of lonesum matrices is that they avoid the permutation matrices of size $2\times 2$. (A matrix \emph{avoids} a $2\times 2$ matrix $B$ if it does not contain it as a submatrix, i.e. there are no row and column indices $i,i',j,j'$ such that
$\left(\begin{array}{cc}
 a_{i,j} & a_{i,j'}\\
a_{i',j} &a_{i',j'} 
\end{array}\right)=B$.)

Ryser \cite{Ryser} showed that this is an alternative characterization of lonesum matrices.
\begin{theorem}(\cite{Ryser})
A $01$ matrix is lonesum if and only if it does not contain any of the two submatrices:
\begin{align*}
\begin{pmatrix}
1& 0\\
0& 1
\end{pmatrix}, 
\begin{pmatrix}
0& 1\\
1& 0
\end{pmatrix}.
\end{align*}
\end{theorem}
One direction of this theorem is obvious: in the case of existence of one of the submatrices above we can switch it to the other one without changing the row and column sum vectors.  

Lonesum matrices have interesting combinatorial properties and they are in bijection with many other combinatorial objects. See \cite{BH1}, \cite{BH2} for an actual list of the related objects.  
As every matrix, lonesum matrix is a sequence of the column vectors. However, since in a lonesum matrix the two permutation matrices are forbidden, there are some constraints on the columns that can exist together in the same matrix. The avoidance of the two forbidden submatrices restricts the structure of these matrices. A lonesum matrix is uniquely determined by a pair of ordered partitions of the sets of column and row indices. Since this bijection is very important for us, we will describe it in details later.  

Brewbaker used these two interpretations to derive combinatorially two formulas for the number of lonesum matrices \cite{Brewbaker}. 
\begin{align*}
\sum_{m=0}^{\min(n,k)}m!{\stirling {n+1} {m+1} }m!{\stirling {k+1} {m+1}}=(-1)^n\sum_{m=0}^n(-1)^mm!{\stirling n m}(m+1)^k
\end{align*}
These numbers are the poly-Bernoulli numbers (with negative $k$ indices). Poly-Bernoulli numbers were introduced analytically in the 1990's by Kaneko \cite{Kaneko}. Since then there were an amount of generalizations and analogous of these numbers investigated.

The wealth of combinatorial properties, connections, and applications motivates us to investigate lonesum matrices in more details. For instance, considering lonesum matrices without any all-zero column revealed the connection to the well-studied research topic of permutation tableaux \cite{BH2}. In this paper we study lonesum matrices with restrictions on the number of columns and rows of the same type. We derive the double exponential function, an exact formula, a recursion and some properties of these matrices. We want to keep the paper self-contained, understandable also for readers who are not familiar in the topic. For this reason, we recall all the necessary informations, relevant identities with proofs, and describe the bijections on that the derivation of the formulas rely.

\section{Enumeration of restricted lonesum matrices}

We will need some terminology. Let $c(M)$ (resp.~ $r(M)$) denote the number of columns (resp.~ the number of rows) of the matrix $M$. We call the \emph{type} of the column (resp.~ row) the column (resp. row) itself, i.e., two columns have different types, if they differ at least at one position. Let $\mathcal{C}(M)$ (resp. $\mathcal{R}(M)$) denote the set of the columns (resp.~ rows) of different types in a lonesum matrix $M$. We call the sum of the entries (equivalently the number of $1$ entries) in a column (resp.~ row) the \emph{weight} of the column (resp.~ row). It can be shown that columns of different types in a lonesum matrix have different weights. Moreover, for any two columns $c_i$ and $c_j$ with $w(c_i)<w(c_j)$ $c_j$ has also $1$ entries at the positions of 1's in $c_i$. This property follows from the characterization of lonesum matrices with forbidden submatrices. Should $c_i$ namely have a $1$ at a position where $c_j$ has a $0$, one of the matrix of the forbidden set would appear. The same argument works for the rows. It is easy -  but important - to see that the number of columns of different type with weight at least one is equal to the number of rows of different types with weight at least one. It is meaningful to order the columns in $\mathcal{C}(M)$ (resp.~ $\mathcal{R(M)}$) according to their weights.  We call the sequence of the different column vectors (resp.~ row vectors) of a lonesum matrix that are ordered by their weight $(c_1,c_2,\ldots c_m)$ (resp.~ $(r_1,\ldots, r_m)$)the \emph{column sequence} (resp.~ \emph{row sequence}) of the matrix. In this order $c_{i+1}$ can be obtained from $c_{i}$ by switching some $0$s to $1$s.

Let $\mathcal{L}_{\leq d}$ denote the set of lonesum matrices without all-zero columns and rows that have at most $d$ columns (resp.~ rows) of the same type. $\mathcal{L}_{\leq d}(n,k)$ denotes the set of lonseum matrices with the properties given above of size $n\times k$. Our main theorem concerns with the enumeration of the set $\mathcal{L}_{\leq d}(n,k)$. Since partitions, in particular restricted partitions, play a crucial role in the next section we recall some facts about partitions and its counting sequence.    

\subsection{Restricted partitions}
The number of partitions of an $n$ element set into $m$ non-empty blocks is given by the Stirling numbers of the second kind ${\stirling n m}$. Restriction on the size of the blocks of the partition leads to the restricted Stirling numbers of the second kind. Partitions with this restrictions arose recently in several problems \cite{Zeilberger,Corcino,Mezo}. Let ${\stirling n m}_{\leq d}$  denote the number of partitions of an $n$ element set into $m$ non-empty blocks such that there is no block that consists of more than $d$ elements. We recall some properties according to \cite{Mezo}. 

Let $P_{m_1,\ldots, m_d}(n)$ denote the number of partitions of $n$ with $m_i$ blocks of size $i$. Then we have 
\begin{align*}
{\stirling n m }_{\leq d}= \sum_{\sum_{m_i=m;\atop {\sum im_i=n}}}P_{m_1,\ldots, m_d}(n).
\end{align*} 
A formula for $P_{m_1,\ldots, m_d}(n)$ can be derived using combinatorial argumentations the following way:
we associate to each permutation of $n$ a partition $P_{m_1,\ldots, m_l}(n)$. The first $m_1$ element are in the blocks of size $d$, the next $2m_2$ elements will form the blocks of size $2$, and so on. For $i$, we cut the $im_i$ elements into blocks the most natural way, the first $i$ element form one block, the next $i$ elements another block and so on. This way we associated the same partition to certain permutations. For instance, let $n=12$, $m_1=2$, $m_2=3$, $m_3=0$, $m_4=1$. Then to both permutations $\pi=8\,5\,2\,10\,3\,7\,6\, 12\, 1\, 4 \, 9\, 11$ and $\sigma=5\,8\,10\,2\, 12\,6\, 3\, 7\, 4\, 9\, 11\, 1$ the partition
\[P_{2,3,0,1}(12)=\{5\}\{8\}\{2,10\}\{3,7\}\{6,12\}\{1,4,9,11\}\]
 will be associated. The class of permutations with the same associated partition contains \[(1!)^{m_1}m_1!(2!)^{m_2}m_2!\cdots (d!)^{m_d}m_d!\] permutations, since permuting the elements inside a block and permuting the blocks of the same size does not influence/change the associated partition. Hence, we have
\begin{align*}
{\stirling n m }_{\leq d}=\sum_{\sum m_i=m; \atop {\sum im_i=n}}\frac{n!}{(1!)^{m_1}m_1!(2!)^{m_2}m_2!\cdots (d!)^{m_d}m_d!}.
\end{align*}
It is clear that if $d>n-m$ then ${\stirling n m}_{\leq d}={\stirling n m}$ and if $m<\lceil\frac{n}{d}\rceil$ then ${\stirling n m}_{\leq d}=0$.

The exponential generating function can be derived using the symbolic method \cite{Flajolet}. In the appendix we recall some ideas, definitions, translation rules of the symbolic method that we use in this paper. 
A partition is a set of sets, so taking the restriction on the number and size of the blocks also into consideration the construction is the following:
\begin{align*}
\mbox{SET}_{m}(\mbox{SET}_{>0,\leq d}(\mathcal{Z})),
\end{align*}
where $\mathcal{Z}$ is the atomic class, $\mbox{SET}_{m}$ denotes the class of $m$-sets, and $\mbox{SET}_{>0,\leq d}$ denotes the class of sets with sets of size greater than $0$ but at most $d$.
The symbolic method allows us to turn this construction immediately into the generating function: 
\begin{align*}
\sum_{n=0}^{\infty}{\stirling n m}_{\leq d}\frac{z^n}{n!}=\frac{(z+\frac{z^2}{2!}+\cdots+\frac{z^d}{d!})^m}{m!}.
\end{align*}
Using the notation introduced in \cite{Mezo}:
\[E_{d}(z)=\sum_{i=0}^d\frac{(z^i)}{i!},\]
the generating function can be written in the form:
\begin{align*}
\sum_{n=0}^{\infty}{\stirling n m}_{\leq d}\frac{z^n}{n!}=\frac{(E_{d}(z)-1)^m}{m!}.
\end{align*}
For calculations of values a recurrence relation could be helpful. 
A partition of $n$ elements can be obtained from a partition of $n-1$ elements by adding a new, the $n$th element to the partition.
It can form a single new block or we can add it to a block of size at least 1. In the first case we can do this in ${\stirling {n-1} {m-1}}_{\leq d}$ ways.
We obtain different kind of recurrence relations as we keep track of different properties of the block into that the $n$th element was added. 
When we just say that we put it into one of the remaining blocks, we have to reduce the $m{\stirling {n-1} m}_{\leq d}$ ways with the number of cases when we destroy the restriction on the size of the blocks. So, what is the number of these "bad" cases? Choose a $d$ element block (to put the $n$th element into this block) in ${n-1\choose d}$ ways and take an arbitrary partition of the remaining $n-d-1$ elements into $m-1$ blocks of size at most $d$.
Hence, 
\begin{align*}
{\stirling n m}_{\leq d}={\stirling {n-1} {m-1}}_{\leq d}+m{\stirling {n-1} m}_{\leq d}-{n-1 \choose d}{\stirling {n-d-1} {m-1}}_{\leq d}
\end{align*} 
We can consider exactly the blocksize of the block into that we the $n$th element add. Choose the block of size $i$ $(i=0,1,\ldots d-1)$, a set of size $i$ out of the $n-1$ elements, and the remaining $n-d-i$ elements will form a restricted partition. 
\begin{align*}
{\stirling n m}_{\leq d}=\sum_{i=0}^{d-1}{n-1 \choose i}{\stirling{n-1-i} {m-1}}_{\leq d}
\end{align*} 
The interested reader can find more identities involving restricted Stirling numbers of the second kind in \cite{Choi}, \cite{Mezo}. 
\subsection{The proof of the main theorem} 
Now we are ready to enumerate restricted lonesum matrices and prove the main theorem.
\begin{theorem}
The number of lonesum matrices without all-zero columns and rows that have at most $d$ columns (resp.~ rows) of the same type is given by the following formula: 
\begin{align*}
|\mathcal{L}_{\leq d}(n,k)|=\sum_{m=0}^{\min(n,k)}m!{\stirling n m}_{\leq d}m!{\stirling k m}_{\leq d}
\end{align*}
\end{theorem}
\begin{proof}
Let $M\in \mathcal{L}_{\leq l}(n,k)$. Let $m$ denote the length of the column (resp.~ row) sequence, i.e. the number of different columns in the matrix $M$. Rearrange the columns and the rows of $M$ according their weights in decreasing order obtaining $M^*$. From the property of the column sequence and row sequence it follows that $M^*$ has a staircase shape. To a horizontal step in this staircase shape belong some indices, the indices of the columns which have their last $1$ entry at the height of this horizontal step. These columns have obviously not only the same weight, but are of the same type. Similarly, to each vertical step belong some indices, the indices of the rows which have their rightmost $1$ entry at the position of the particular vertical step.
In other words, to the columnsequence $(c_1,c_2,\ldots c_m)$ we can associate an ordered partition of $n$ into $m$ blocks that we obtain the following way: the $i$th block consists of the indices of the positions in that $c_{i-1}$ and $c_{i}$ differs, the positions where $c_{i-1}$ has $0$s and $c_{i}$  has $1$s. The first block contains the positions of $1$ entries of the column $c_1$. Analogously, to the row sequence we associate an ordered partition of $k$ into $m$ non-empty blocks.     
Hence, $M^*$ is uniquely determined by two ordered set partitions of the index sets $\{1,\ldots, n\}$ and $\{1,\ldots, k\}$ both into $m$ non-empty blocks. Since we have the restriction on the columns and rows, the blocks of the partitions consist of at most $d$ elements.
\end{proof}
\begin{example}
Consider the matrix $M$:
\[M=\begin{pmatrix}
0&0&0&0&0&0&0&0&1\\
1&0&0&0&0&1&1&0&1\\
1&0&0&1&0&1&1&1&1\\
1&1&1&1&1&1&1&1&1\\
0&0&0&0&0&0&0&0&1\\
1&0&0&1&0&1&1&1&1\\
1&1&1&1&1&1&1&1&1\\
1&0&0&1&0&1&1&1&1
\end{pmatrix}\]
\begin{figure}[h]
\begin{center}
\begin{picture}(10,8)
\put(0,0){\line(1,0){9}}
\put(0,8){\line(1,0){9}}
\put(0,0){\line(0,1){8}}
\put(9,0){\line(0,1){8}}
\thicklines
\put(0,0){\line(1,0){1}}
\put(1,0){\line(0,1){2}}
\put(1,2){\line(1,0){3}}
\put(4,2){\line(0,1){1}}
\put(4,3){\line(1,0){2}}
\put(6,3){\line(0,1){3}}
\put(6,6){\line(1,0){3}}
\put(9,6){\line(0,1){2}}
\put(0.4,0.2){${ 9}$}
\put(1.2,0.3){${ 1}$}
\put(1.2,1.3){${ 5 }$}
\put(1.4,2.2){${ 1}$}
\put(2.4,2.2){${ 6}$}
\put(3.4,2.2){${ 7}$}
\put(4.2,2.3){${ 2}$}
\put(4.3,3.2){${ 4}$}
\put(5.3,3.2){${ 8}$}
\put(6.2,3.3){${ 3}$}
\put(6.2,4.3){${ 6}$}
\put(6.2,5.3){${ 8}$}
\put(6.3,6.2){${ 2}$}
\put(7.3,6.2){${ 3}$}
\put(8.3,6.2){${ 5}$}
\put(9.2,6.3){${4}$}
\put(9.2,7.3){${ 7}$}
\put(2.5,5){\Huge{1}}
\put(6.5,1){\Huge{0}}
\end{picture}
\end{center}
\caption{The rearranged matrix $M^*$ and the associated set partitions}
\end{figure}

The two ordered partitions that are associated to the matrix $M$ are $9|176|48|235$ and $15|2|368|47$.
\end{example}

\begin{corollary}
Let $\mathcal{L}_{r\leq d}(n,k)$, (resp.~ $\mathcal{L}_{c\leq d}(n,k)$) denote the set of lonseum matrices of size $n\times k$ without all-zero columns and rows such that the number of rows (resp.~the columns) of the same type is at most $d$, while there is no restriction on the number of columns (resp.~rows) of the same type. Then we have
\begin{align*}
|\mathcal{L}_{r\leq d}(n,k)|&=\sum_{m=0}^{\min{(n,k)}}m!{\stirling n m}_{\leq d}m!{\stirling k m},\\
|\mathcal{L}_{c\leq d}(n,k)|&=\sum_{m=0}^{\min{(n,k)}}m!{\stirling n m}m!{\stirling k m}_{\leq d}. 
\end{align*}
\end{corollary}

Let $L_{\leq d}(x,y)$ be the double exponential generating function of the lonesum matrices that contain at most $d$ columns and rows of the same type and every row and column contains at least one 1 entry. 
 
\begin{align*}
L_{\leq d}(x,y)=\sum_{M\in \mathcal{L}_{\leq d}}\frac{x^{r(M)}}{r(M)!}\frac{y^{c(M)}}{c(M)!}.
\end{align*}
\begin{theorem} We have
\[
L_{\leq d}(x,y)=\frac{1}{E_{d}(x)+E_{d }(y)-E_{d}(x)E_{d}(y)}.
\]
\end{theorem}
\begin{proof}
Ordered pairs of restricted partitions can be viewed as a sequence of pairs of sets of size at least 1 and at most $d$ of an $[n]$ and a $[k]$ element set. Therefore, our construction is:
\begin{align*}
\mbox{SEQ}\left(\mbox{SET}_{>0,\leq d}(\mathcal{X})\times \mbox{SET}_{>0,\leq d}(\mathcal{Y})\right),
\end{align*} 
where $\mathcal{X}$ and $\mathcal{Y}$ are atomic classes. 
The symbolic method \cite{Flajolet} tells us that this construction translates into the generating function: 
\begin{align*}
L_{\leq d}(x,y)=\frac{1}{1-(E_{d}(x)-1)(E_{d}(y)-1)}.
\end{align*}
After simplification we obtain the theorem.
\end{proof}
The correspondence that was described in the proof of the main theorem (Theorem 2). implies that the matrices of the set $\mathcal{L}_{\leq d}$ can be characterized an alternative way. 
\begin{theorem} Let $M$ be a $01$ matrix. $M$ belongs to the set $\mathcal{L}_{\leq d}(n,k)$ if and only if the following three properties hold: 
\begin{itemize}
\item[i.] every column  and row consists at least one $1$ entry, i.e. the matrix does not contain any all zero columns and all zero rows, 
\item[ii.] two consecutive columns in the column sequence differ only at most by $d$ entries, i.e. $w(c_{i+1})-w(c_i)\leq d$ for all $i=1,\ldots, m-1$, where $m$ is the length of the column sequence,
\item[iii.]two consecutive rows in the row sequence differ only at most by $d$ entries, i.e. $w(r_{i+1})-w(r_i)\leq d$ for all $i=1,\ldots, m-1$, where $m$ is the length of the row sequence.
\end{itemize}
\end{theorem}
\begin{proof}
The theorem follows from the bijection given in the proof of theorem 2. 
\end{proof}

\begin{remark}
Poly-Bernoulli numbers of negative k indices enumerate the lonesum matrices of size $n\times k$.  In \cite{Mezo} the authors defined the incomplete poly-Bernoulli numbers analogously to the poly-Bernoulli numbers using restricted Stirling numbers of the second kind. This analytical definition does not lead to positive integers and does not count the lonesum matrices with the restriction given above.  
\end{remark}
However, because of the strong connection to poly-Bernoulli numbers we introduce for the number of the set $\mathcal{L}_{\leq d}(n,k)$ the notation $pB(n,k)_{\leq d}$.

The following two properties of $pB(n,k)_{\leq d}$ is immediate.
\begin{itemize}
\item[i.] The numbers $pB(n,k)_{\leq d}$ are in the parameters $n$ and $k$ symmetric: \[pB(n,k)_{\leq d} =pB(k,n)_{\leq d}\]
\item[ii.] $\mathcal{L}_{\leq d}(n,k)$ is not empty if and only if $\lceil{\frac{n}{d}}\rceil\leq k\leq nd $, i.e.
\[pB(n,k)_{\leq d}=0,\quad \mbox{if} k<\lceil{\frac{n}{d}}\rceil \quad \mbox{or} \quad k> nd\]
\end{itemize}
\begin{proof}
(i.) It is obvious from the combinatorial definition.

(ii.) The common length of the column sequence and row sequence, $m$ has to satisfy:
\[\max\left({\lceil\frac{n}{d}}\rceil,\lceil\frac{k}{d}\rceil\right)\leq m\leq \min(n,k).\]
\end{proof}
Next we derive a recursive relation, though this recursion is too complicated for practical use. It shows only the simple recursive structure of these matrices.   
\begin{theorem}The following recursive relation holds for the number of restricted lonesum matrices, $pB(n,k)_{\leq d}$, for $n\geq 1$ and $k\geq 1$:
\begin{align*}
pB(n,k)_{\leq d}=\sum_{n_i,k_j\in \{1,\ldots d\}\atop  i,j\in\{1,2\}} {n\choose n_1,n_2}{k\choose k_1,k_2} pB(n-(n_1+n_2),k-(k_1+k_2))_{\leq d},
\end{align*}
where ${n\choose n_1,n_2}$ denotes the multinomial coefficient.
\end{theorem}
\begin{proof}
Let $M$ be a matrix in $\mathcal{L}_{\leq d}(n,k)$ with column sequence $(c_1,\ldots c_m)$ and row sequence $(r_1,\ldots r_m)$. $c_m$ is an all-$1$ row, that differs from $c_{m-1}$ by at most $d$ positions. Similarly, for rows. Delete from the matrix $M$ the all-1 columns and rows obtaining $\widehat{M}$. Clearly, the column and row sequence of $\widehat{M}$ is obtained from the column and row sequence of $M$ by deleting $c_m$, $r_m$, $r_1$, and $c_1$. Let  $n_1=w(r_1)$, $k_1=w(c_1)$,  $n_2=w(r_{m})-w(r_{m-1})$, $k_2=w(c_{m})-w(c_{m-1})$. The $\widehat{M}$ is an arbitrary matrix of the set $\mathcal{L}_{\leq d}(n-n_1-n_2,k-k_1-k_2)$.
The positions of $1$ in $r_1$,$c_1$ can be choosen in ${n\choose n_1}$ resp. ${k\choose k_1}$ ways, and the positions of $1$ entries  in $c_m$, $r_m$ that are not in $c_{m-1}$ resp.~$r_{m-1}$ can be choosen in  ${n-n_1\choose n_2}$ resp.~${k-k_1\choose k_2}$ ways.   
\end{proof}
For the special case $d=1$ only the square matrix $n\times n$ exists and the theorem reduces to the fact that the number of lonesum matrices of size $n\times n$ where each column and each row differs is $(n!)^2$. It is maybe interesting that indeed such a lonesum matrix encodes a pair of permutations of $[n]$. 
\begin{example}The two permutations (written in one-line notation): $\pi=2471365$ and $\sigma=3156724$ is coded in the matrix $P$.
\begin{align*}
\begin{pmatrix}
1&0&1&1&1&1&1\\
0&0&0&0&1&1&0\\
1&1&1&1&1&1&1\\
0&0&0&0&1&0&0\\
1&0&1&0&1&1&1\\
1&0&1&0&1&1&1\\
1&0&1&0&1&1&0\\
0&0&1&0&1&1&0\\
0&0&0&0&1&0&0
\end{pmatrix}
\end{align*}
\end{example}

The special case $d=2$ may be of interest because it involves the Bessel numbers, the coefficients of the Bessel functions. An explicit formula is 
\[B(n,k)=\frac{n!}{2^{n-k}(2k-n)!(n-k)!}.\]
$B(n,k)$ \cite{OEIS} counts the number of involutions of $n$ with $k$ pairs. Equivalently, the $k$-matchings (matching with $k$ edges) of the complete graph $K_n$. We give in the next table the values of $pB(n,k)_{\leq 2}$ for small $n$ and $k$, i.e. the number of lonesum matrices in that a given type of columns (or rows) appear at most two times. 

\begin{table}[h]
\begin{center}
\begin{tabular}{|c|cccccc|}
\hline
n/k & 1 & 2 & 3 & 4 & 5 & 6 \\\hline
1 & 1 & 1 & &&&\\
2 & 1 & 4 & 12 & 12 & &\\
3& & 12 & 72 & 252 & 540 & 540 \\
4 & &12 & 252 & 1908 & 9000 & 2916\\
5 & & & 540 & 9000 & 80100 & 483300\\
6 & & & 540 & 29160 & 483300 & 4932900\\\hline   
\end{tabular}
\end{center} 
\caption{$pB(n,k)_{\leq 2}$}
\end{table}
\begin{remark}
The recent author investigates with J. L. Ramírez lonesum matrices with further restrictions on the number of columns and rows of the same type in a forthcoming paper \cite{Rami}. 
\end{remark}

\section{Appendix}
The symbolic method \cite{Flajolet} is a powerful method in modern combinatorics for the determination of generating functions. 
It develops a systematic translation mechanism between combinatorial constructions and operations on generating functions. This translation process is purely formal. A list of the translation rules of the core constructions translates into generating functions or into equations relating generating functions was given in \cite{Flajolet}. These basic constructions enable us to take into account structures that would be hard to deal with otherwise. 

A \emph{combinatorial class} is a set of objects with a notion of size attached so that the number of objects of each size in $\mathcal{A}$ is finite. A \emph{counting sequence} of a combinatorial class is the sequence of integers $(A_n)_{n\geq 0}$ where $A_n$ is the number of objects in class $\mathcal{A}$ that have size $n$. The exponential generating function of a class $\mathcal{A}$ is 
\begin{align*}
A(z)=\sum_{n\geq 0}A_n\frac{z^n}{n!}=\sum_{\alpha\in \mathcal{A}}\frac{z^{|\alpha|}}{|\alpha|!}.
\end{align*} 
It is said that the variable $z$ \emph{marks} the size in the generating function. $A_n=n![z^n]A(z)$, where $[z^n]$ denotes the coefficient of $z^n$ in the function $A(z)$.

An object is \emph{labeled} if each atom bears an integer label, in such a way that all the labels occurring in an object are distinct. A labeled object may be relabeled in a consistent way: the order relation among labels has to be preserved. 

The neutral object $\epsilon$ has size $0$ and has no label: \[\mathcal{E}=\{\epsilon\}\rightarrow E(z)=1. \]
The atomic class $\mathcal{Z}$ is formed of a unique object of size one and label 1:
 \[\mathcal{Z}=\{1\}\rightarrow \mathcal{Z}(z)=z.\]
\begin{itemize}
\item[$\mathcal{B}\cup\mathcal{C}$] The union of two classes consists of objects that are objects of the class $\mathcal{B}$ or of the class $\mathcal{C}$. The translation rule is:
\begin{align*}
\mathcal{A}=\mathcal{B}\cup\mathcal{C}\rightarrow A(z)=B(z)+C(z)
\end{align*}
\item[$\mathcal{B}\times\mathcal{C}$] The labeled product of $\mathcal{B}$ and $\mathcal{C}$ is obtained by forming pairs $(\beta,\gamma)$ with $\beta \in \mathcal{B}$, $\gamma\in \mathcal{C}$ and performing all possible order-consistent relabelings. The translation rule is:
\begin{align*}
\mathcal{A}=\mathcal{B}\times\mathcal{C}\rightarrow A(z)=B(z)C(z)
\end{align*}
\item[$\mbox{SEQ}_k(\mathcal{B})$] The $k$th power of $\mathcal{B}$ is defined as $(\mathcal{B}\times \cdots \times\mathcal{B})$ with $k$ factors equal to $\mathcal{B}$. The translation rule is:
\begin{align*}
\mathcal{A}=\mbox{SEQ}_(\mathcal{B})\rightarrow A(z)=B(z)^k
\end{align*}
\item[$\mbox{SEQ}(\mathcal{B})$] The sequence class of $\mathcal{B}$ is defined by 
\[\mbox{SEQ}(\mathcal{B}):=\{\epsilon\}\cup \mathcal{B}\cup \mathcal{B}\times \mathcal{B}\cup\cdots=\bigcup_{k\geq 0}\mbox{SEQ}_k(\mathcal{B}).\]The translation rule is:
\begin{align*}
\mathcal{A}=\mbox{SEQ}(\mathcal{B})\rightarrow A(z)=\sum_{k=0}^{\infty}B(z)=\frac{1}{1-B(z)}
\end{align*}
\item[$\mbox{SET}_k(\mathcal{B})$] The class of $k$-sets formed from $\mathcal{B}$ is a $k$-sequence modulo the equivalence relation that two sequences are equivalent when the components of one are the permutation of the components of the other. The translation rule is:
\begin{align*}
\mathcal{A}=\mbox{SET}_k(\mathcal{B})\rightarrow A(z)=\frac{1}{k!}B(z)^k
\end{align*}
\item[$\mbox{SET}(\mathcal{B})$] The class of sets formed from $\mathcal{B}$ is defined by
\[\mbox{SET}(\mathcal{B})=\{\epsilon\}\cup \mathcal{B} \cup \mbox{SET}_2(\mathcal{B})\cup \cdots =\bigcup_{k\geq 0}\mbox{SET}_k(\mathcal{B})\]
The translation rule is:
\begin{align*}
\mathcal{A}=\mbox{SET}(\mathcal{B})\rightarrow A(z)=\sum_{k=0}^{\infty} \frac{1}{k!}B(z)^k=e^{B(z)}
\end{align*}
\end{itemize}

A set partition is a set of sets: $\mbox{SET}(\mbox{SET}_{\geq 1}(\mathcal{Z}))$ and a surjection (ordered partition) is a sequence of sets: $\mbox{SEQ}(\mbox{SET}_{\geq 1}(\mathcal{Z}))$.

\end{document}